\documentclass [final,11pt]{article}
\usepackage{amsfonts,amsmath,amsthm,amscd,amssymb,enumerate,float,latexsym,cite,verbatim,texdraw,floatflt,pb-diagram,secdot,subfigure}
\usepackage[T2A]{fontenc}
\usepackage[cp1251]{inputenc}
\usepackage[russian,english]{babel}
\usepackage{graphicx}
\usepackage{color}
\usepackage[colorlinks]{hyperref}
\numberwithin{equation}{section}
\usepackage[font=small,labelfont=bf,tableposition=top,format=hang,parskip=5pt]{caption}
\title{``Predator and prey'' model revisited - influence of external fluxes and noise.\thanks{Authors are grateful to Catherine Buryachenko and to Vitaliy Slynko for fruitful discussions. Authors appreciate the assistance of Sergei Abakumov who double-checked some numeric results. Work was supported by the Ministry of Education and Science of Ukraine (project 0118U003861)}}
\author{Yaroslav Huriev, Andriy Gusak}
\date{}
\newcommand{\keywords}{\textbf{Key words. }}
\newcommand{\subjclass}{\textbf{MSC 2010. }}
\renewcommand{\abstract}{\textbf{Abstract. }}
\newtheorem{theorem}{Theorem}

\begin{document}
	\maketitle
	\begin{abstract}
		Well-known predator-prey model is modified in two ways. First, regular adding or regular deleting of preys or/and predators is considered. Steady-state and stability diagram is found.  Second, random fluctuations of the birthrate and other kinetic coefficients are studied – parabolic law of random walk in (X,Y)-space is found and proved for small deviations from steady-state.
	\end{abstract}\medskip
	\subjclass{34A34, 34F05, 34D08}	\medskip\\
	\keywords{Lotka-Volterra model, nonlinear differential equations, stability analysis, noise}
	\medskip
	\section{Introduction}
	Predator-prey model, introduced by Lotka and Volterra, is a basic synergetic model demonstrating oscillatory behavior of the nonlinear biological, chemical or economic systems \cite{lit1,lit2,lit3}. It is governed by the simple set of two nonlinear equations taking into account the natural birthrate $ \overline{k}_{1} $  of preys, natural death rate $ \overline{k}_{4} $ of predators, as well as  ``collisions'' of preys with predators, unlucky for preys (rate $ \overline{k}_{2}$) and lucky for feeding the new generations of predators (rate $\bar{k}_{3}$ ):
	\begin{align}
	&\frac{d\bar{X}}{d\bar{t}}=\bar{k}_1\bar{X}-\bar{k}_2\bar{X}\bar{Y}\nonumber\\
	&\frac{d\bar{Y}}{d\bar{t}}=\bar{k}_3\bar{X}\bar{Y}-\bar{k}_4\bar{Y}
	\end{align}
	(X - number of preys, Y - number of predators, ). If the ``CREATOR'' of this ecosystem, choosing the initial numbers of both species,``misses'' the stationary numbers $ \bar{X}^{st}=\bar{k}_4/\bar{k}_3,\bar{Y}^{st}=\bar{k}_1/\bar{k}_2 $, system demonstrates the oscillatory behavior, all oscillations proceeding around the mentioned stationary point. Transition to non-dimensional variables,  $X=\bar{X}/\bar{X}^{st},Y=\bar{Y}/\bar{Y}^{st},t=\sqrt{\bar{k}_1\bar{k}_4}\bar{t}$ ,provides the set of two universal equations:
	\begin{align}
	&\frac{1}{a}\frac{dX}{dt}={k}_1{X}-{k}_2{X}\cdot{Y}\nonumber\\
	&a\frac{d{Y}}{d{t}}={k}_3{X}\cdot{Y}-{k}_4{Y}
	\end{align}
	with $a=\sqrt{\bar{k}_1/\bar{k}_4},{k}_1=1,{k}_2=1,{k}_3=1,{k}_4=1$. This model is very idealized and almost closed (excluding unlimited feed for preys). It is a main reason why the phase trajectories in the standard Lotka-Volterra model neither converge nor diverge, but just oscillate. Mathematically, it means that the linearized set of equations for deviations,
	\begin{align}
	&\frac{d\delta{X}}{dt}=a\cdot((1-Y^{st})\delta{X}-X^{st}\cdot \delta{Y})=0\cdot{\delta{X}}+(-a)\cdot\delta{Y}\nonumber\\
	&\frac{d{Y}}{d{t}}=\frac{1}{a}(Y^{st}\delta{X}+(X^{st}-1)\cdot\delta{Y})=\frac{1}{a}\cdot\delta{X}+0\cdot\delta{Y}
	\end{align}
	provides the purely imaginary (with zero real part) roots of the characteristic equation:
	$\det 
	\begin{Vmatrix}
		-\lambda& -a\\
		1/a& -\lambda
	\end{Vmatrix}=0\Rightarrow\lambda^{2}=-1\Rightarrow\lambda=\pm{i}$, which means oscillations without divergence or convergence.
	
	Everywhere below we will limit ourselves to the particular, most symmetric case $a=1$.
	
	There exist a lot of modifications and generalizations of predator-prey model \cite{lit4,lit5,lit6,lit7,lit8}, including non-homogeneity of the system and account of diffusion, several types of preys or/and predators, noise (fluctuations) of preys or/and predators numbers. In this paper, we suggest two more ways of modifications, which seem natural:
	\begin{enumerate}
	\item We will introduce sponsors/hunters of predators or/and preys with ``license'' for constant rate of sponsoring/hunting (regular external fluxes).
	\item We will introduce noise of kinetic coefficients.
	\end{enumerate}
	
	We will see that incorporation of adding or deleting predators or/and preys substantially broadens the spectrum of possible regimes: (I) system may remain eternally oscillating without convergence neither divergence, as in the classic LV-model, (II) system can be stable and converge to the steady-state limit, (III) system can be metastable, converging to steady-state from the initial positions in some critical vicinity of stationary solution and diverging from the positions outside this critical region, (IV) totally unstable, always diverging system.
	
	We will also see that the noise of kinetic coefficients, in average, leads to divergence, but the time law for the growth of mean squared distance from the stationary solution is peculiar and resembles Brownian motion.
	
	\section{Influence of regular adding/hunting}
	The basic equations for LV-system (Lotka-Volterra system) with external fluxes are:
	\begin{align}\label{eq4}
	&\frac{dX}{dt}={X}-{X}\cdot{Y}+b_{x}\nonumber\\
	&\frac{d{Y}}{d{t}}={X}\cdot{Y}-{Y}+b_{y}
	\end{align}
	We will start from some simple examples. Regular adding preys without predators being touched ($b_{x}>0,b_{y}=0$) stabilizes the system (Fig.\ref{image1}a), regular hunting preys without predators being touched ($b_{x}<0,b_{y}=0$) destabilizes the system (Fig.\ref{image1}b):
	\begin{figure}[H]
		\subfigure[$b_{x}=0.1,b_{y}=0$]{\includegraphics[width=9cm,height=6cm]{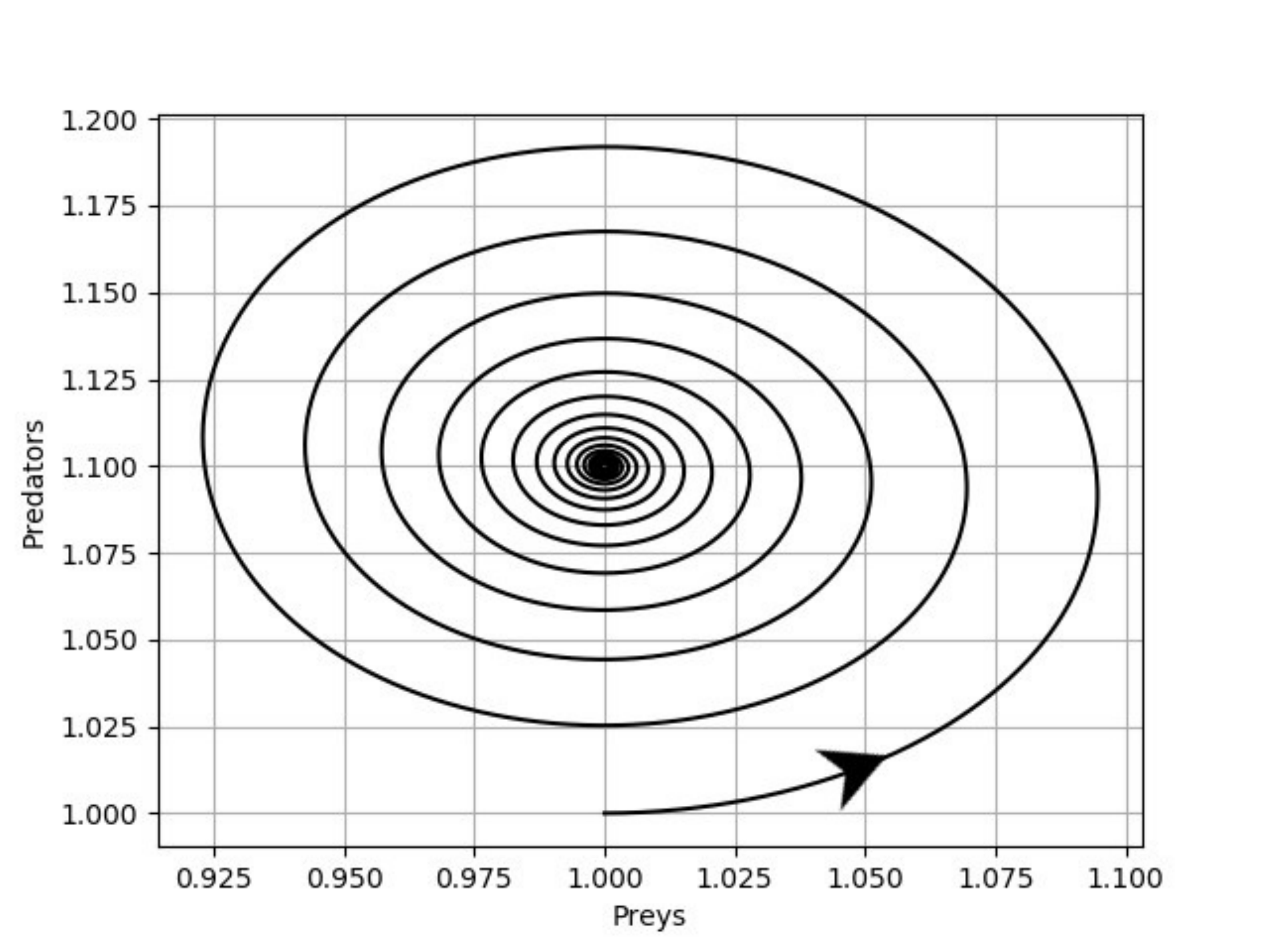}}
		\hfill
		\subfigure[$b_{x}=-0.1,b_{y}=0$]{\includegraphics[width=9cm,height=6cm]{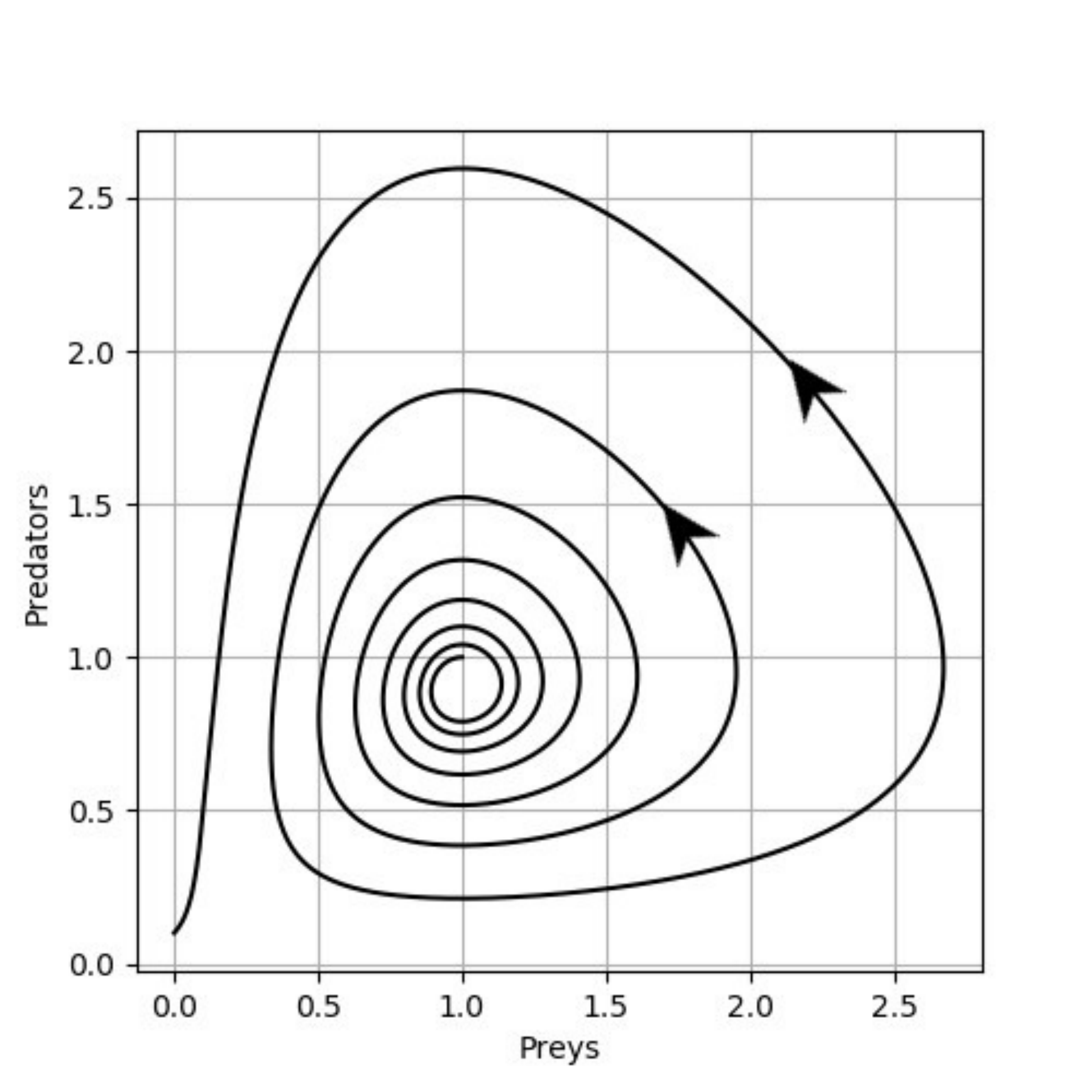}}
		\caption{Phase trajectories for cases of adding and hunting of preys.}
		\label{image1}
	\end{figure}
	
	In cases of adding/hunting only predators without preys being touched, we have analogic situation (Fig.\ref{image2}), but adding of predators stabilizes the system only for $0<b_{y}<1$.
	\begin{figure}[H]
		\subfigure[$b_{x}=0,b_{y}=0.1$]{\includegraphics[width=9cm,height=9cm]{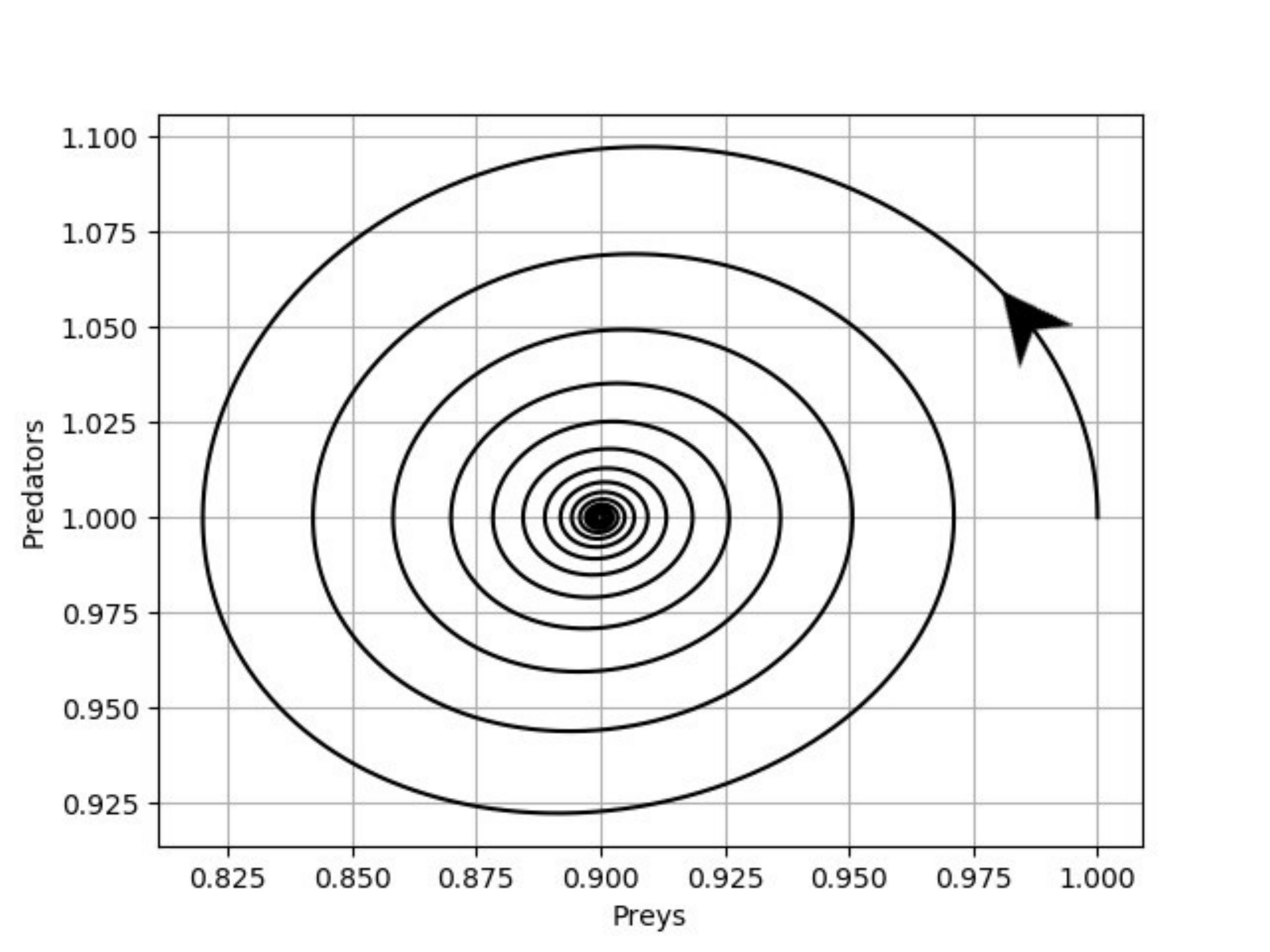}}
		\hfill
		\subfigure[$b_{x}=0,b_{y}=-0.1$]{\includegraphics[width=9cm,height=9cm]{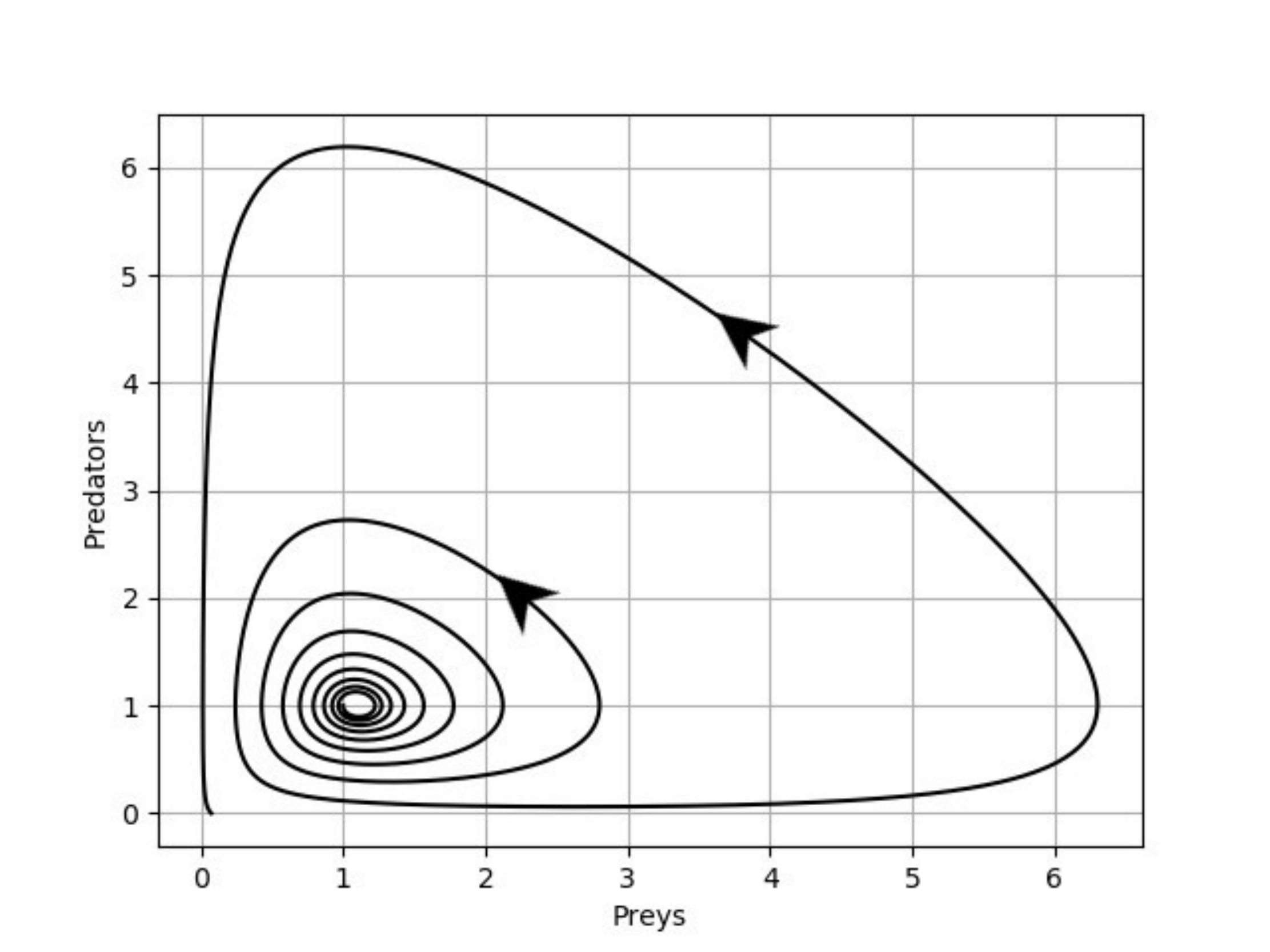}}
		\caption{Phase trajectories for cases of adding and hunting of predators.}
		\label{image2}
	\end{figure}
	
	Now we will consider the general case. Equations for stationary points give the formal solution
	\begin{eqnarray}\label{eq5}
	X^{st}=\frac{1-b_{x}-b_{y}}{2}+\sqrt{\frac{(1-b_{x}-b_{y})^{2}}{4}+b_{x}}, Y^{st}=\frac{1+b_{x}+b_{y}}{2}+\nonumber\\
	+\sqrt{\frac{(1-b_{x}-b_{y})^{2}}{4}+b_{x}}
	\end{eqnarray}
	
	Of course, only positive solutions ($X^{st}>0,Y^{st}>0$) should be further considered. Moreover, these solutions should be at least locally stable. To find the criteria of local stability, we linearize Eq. \eqref{eq4} in the vicinity of stationary solution determined by Eq. \eqref{eq5}:
	\begin{align}
	&\frac{d\delta{X}}{dt}=(1-Y^{st})\delta{X}-X^{st}\cdot\delta{Y}\nonumber\\
	&\frac{dY}{dt}=Y^{st}\delta{X}+(X^{st}-1)\cdot\delta{Y}
	\end{align}
	
	Then the local stability is determined by the real parts of the roots of characteristic equation
	\begin{eqnarray}
	\det 
	\begin{Vmatrix}
	1-Y^{st}-\lambda& -X^{st}\\
	Y^{st}& X^{st}-1-\lambda
	\end{Vmatrix}=0\Rightarrow\lambda^{2}-(X^{st}-Y^{st})\lambda+X^{st}+\nonumber\\
	+Y^{st}-1=0,\nonumber
	\end{eqnarray}
	
	\begin{align}
		&\nonumber\lambda_{1}=\frac{1}{2}((X^{st}-Y^{st})+\sqrt{(X^{st}-Y^{st})^{2}+4(1-X^{st}-Y^{st})})\\
		&\lambda_{2}=\frac{1}{2}((X^{st}-Y^{st})-\sqrt{(X^{st}-Y^{st})^{2}+4(1-X^{st}-Y^{st})})
		\label{eq2.4}
	\end{align}

	The result of steady-state analysis is summarized at Fig.\ref{image3}.
	\begin{figure}[H]
		\center{\includegraphics[width=12cm,height=7cm]{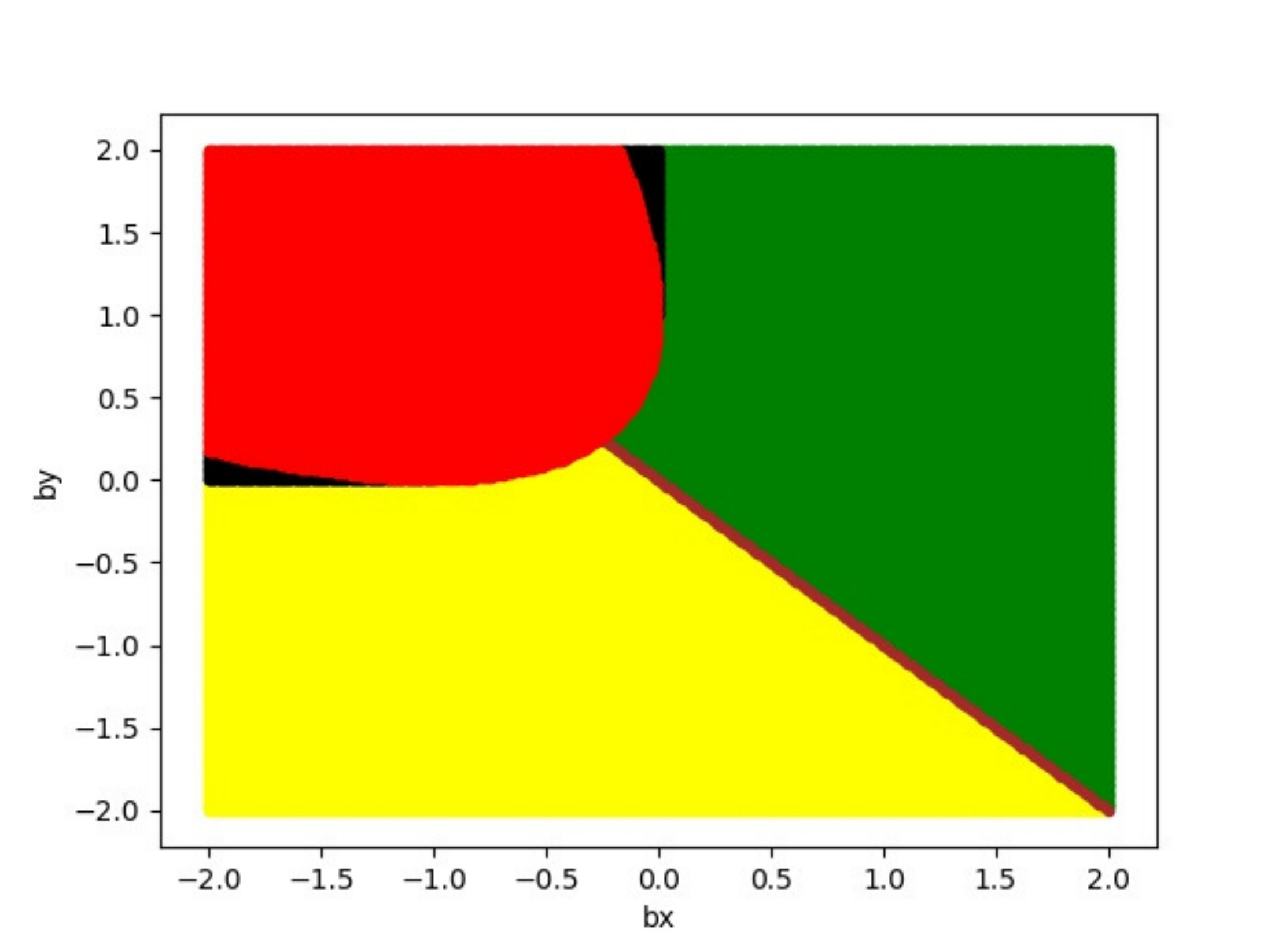}}
		\caption{Diagram ($b_{x},b_{y}$) of regimes in the open Lotka-Volterra system.\\
		Red - negative discriminant,$\frac{(1-b_{x}-b_{y})^{2}}{4}+b_{x}<0$ - no stationary solutions at all.\\
		Black - negative or zero values of $X^{st}$ or $Y^{st}$ - non-physical stationary solutions.\\
		Yellow - physical stationary solutions, but unstable (at least one of the real parts of $\lambda$ is positive or zero).\\
		Brown - oscillatory behavior without convergence neither divergence (zero stability)\\
		Green - locally stable stationary solutions.}
		\label{image3}
	\end{figure}
	
	In general, we should distinguish globally stable and locally stable (metastable) stationary solutions. It will be done elsewhere.
	
	Two examples are demonstrated at Fig. \ref{image4}(a, b). Namely, we demonstrate the convergence region (green) at X-Y plane around the metastable stationary points corresponding to following cases: (a)$b_{x}=1.1,b_{y}=-1.09,X^{st}=1.65,Y^{st}=1.66$, (b)$b_{x}=-0.1,b_{y}=0.3,X^{st}=0.64,Y^{st}=0.84$
	\begin{figure}[H]
		\subfigure[$b_{x}=1.1,b_{y}=-1.09,X^{st}=1.65,Y^{st}=1.66$]{\includegraphics[width=9cm,height=6cm]{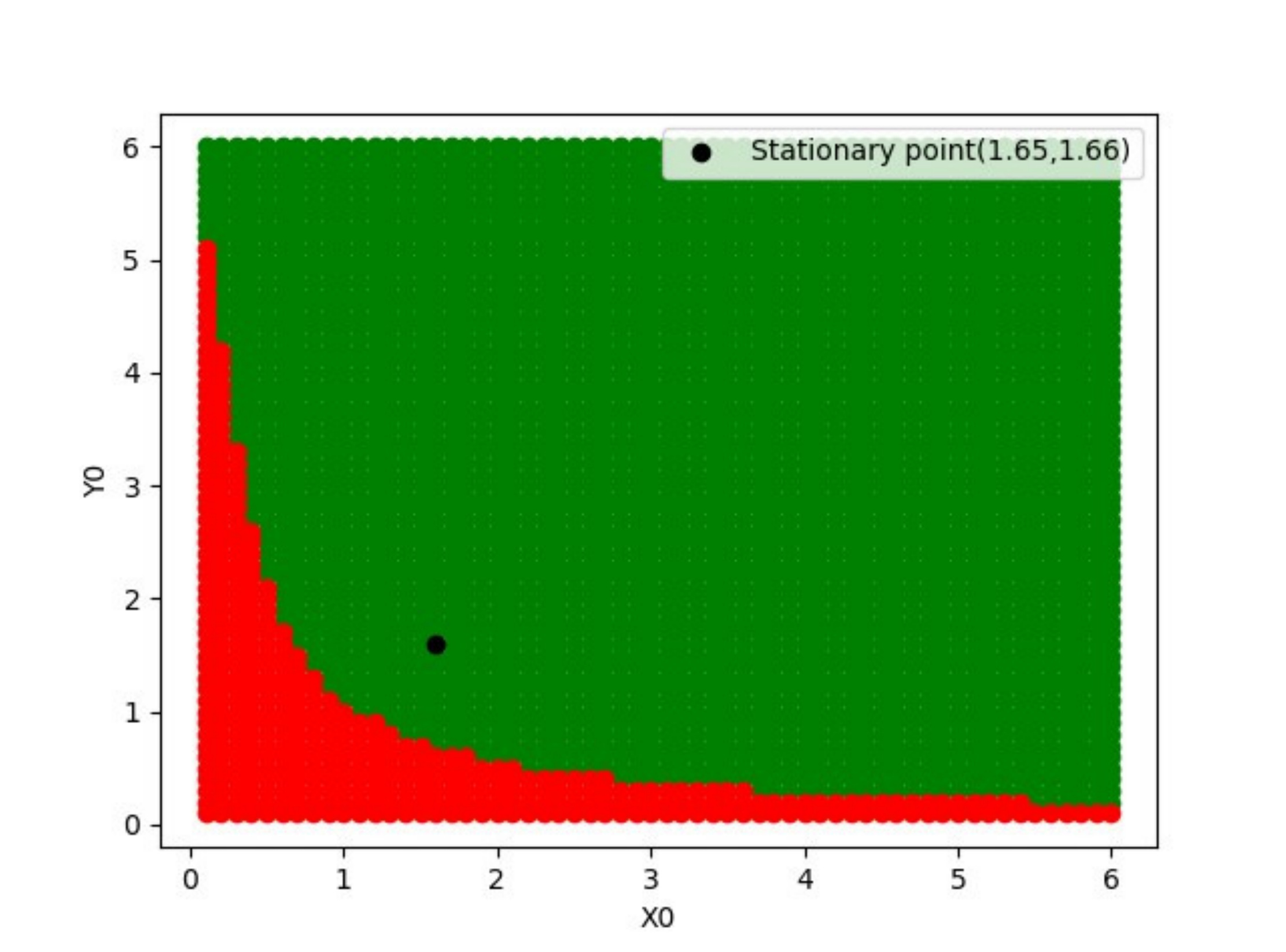}}
		\hfill
		\subfigure[$b_{x}=-0.1,b_{y}=0.3,X^{st}=0.64,Y^{st}=0.84$]{\includegraphics[width=9cm,height=6cm]{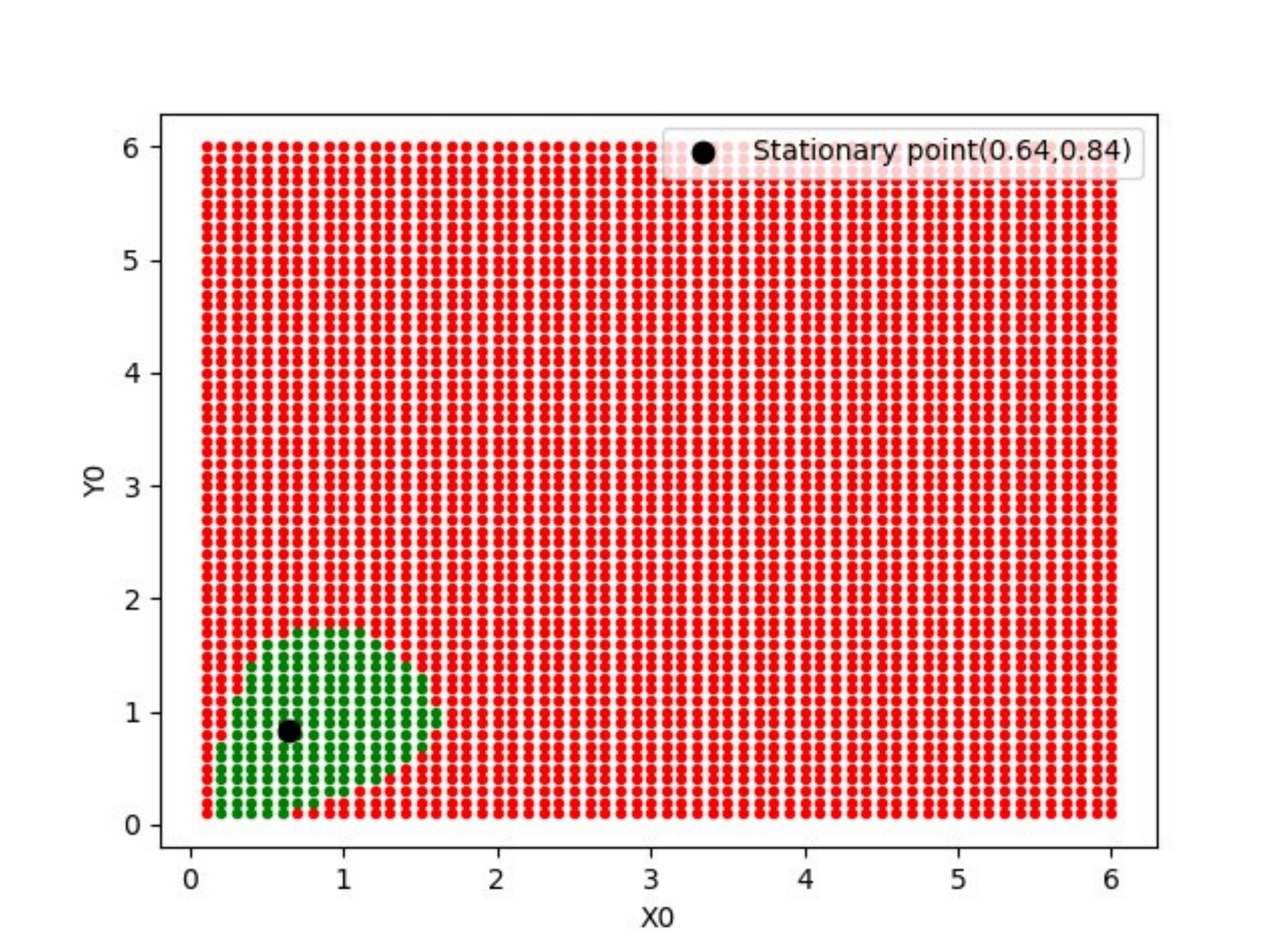}}
		\caption{Stability/collapse diagram
		(stationary point is marked with black spot, convergence region is green)}
		\label{image4}
	\end{figure}

	\section{Noise of kinetic coefficients}
	
	Influence of noise on the behavior of the LV system is not a new problem. Yet, as far as we know, only noise of X and Y had been explored - random adding or deleting of small number of preys or predators\cite{lit4,lit7}. Additionally of this, in our paper we explore the random fluctuations of kinetic coefficients $k_{1}, k_{2}, k_{3}, k_{4}$ (for example, fluctuation of birthrate due to rainy days).
	
	We will start from the Langevine noise of reduced birthrate without memory and with fixed amplitude $A$:
	\begin{align}
	&k_{1}=1+\xi(t)\nonumber\\
	&<\xi(t)\xi(t^{\prime})>=A^{2}\delta(t-t^{\prime})
	\end{align}
	 In case of numeric modeling, one should introduce this noise in such a way that the change of the time step should not change the impact of noise. Our suggestion was a step-wise probability distribution:
	\begin{align}\tag{3.2a}
	k_{1}=1+\frac{A}{\sqrt{dt}}\sqrt{3}(2\cdot{random}-1),
	\end{align}
	
	(Mean square of random function $\sqrt{3}(2\cdot{random}-1)$  is equal to 1). Alternatively, one may use the Gaussian distribution
	\begin{align}\tag{3.2b}
	k_{1}=1+\frac{A}{\sqrt{dt}}\sin ({2\pi{random}})\sqrt{2\ln ({1/random})}
	\end{align}.

	In more details, the introduction of the noise of kinetic coefficients is discussed in \cite{lit8,lit9} for the case of atomic migration.
	We started from stationary point ($X=1, Y=1$) as an initial condition. Typical phase trajectory as a numeric solution of the set
	\begin{align}
	&\nonumber\frac{dX}{dt}=(1+\xi(t))X-X\cdot{Y}\\
	&\tag{3.3}\frac{dY}{dt}=X\cdot{Y}-Y
	\end{align}
	is shown at Fig.\ref{image5}.
	\begin{figure}[H]
		\centering
		\includegraphics[width=8cm,height=8cm]{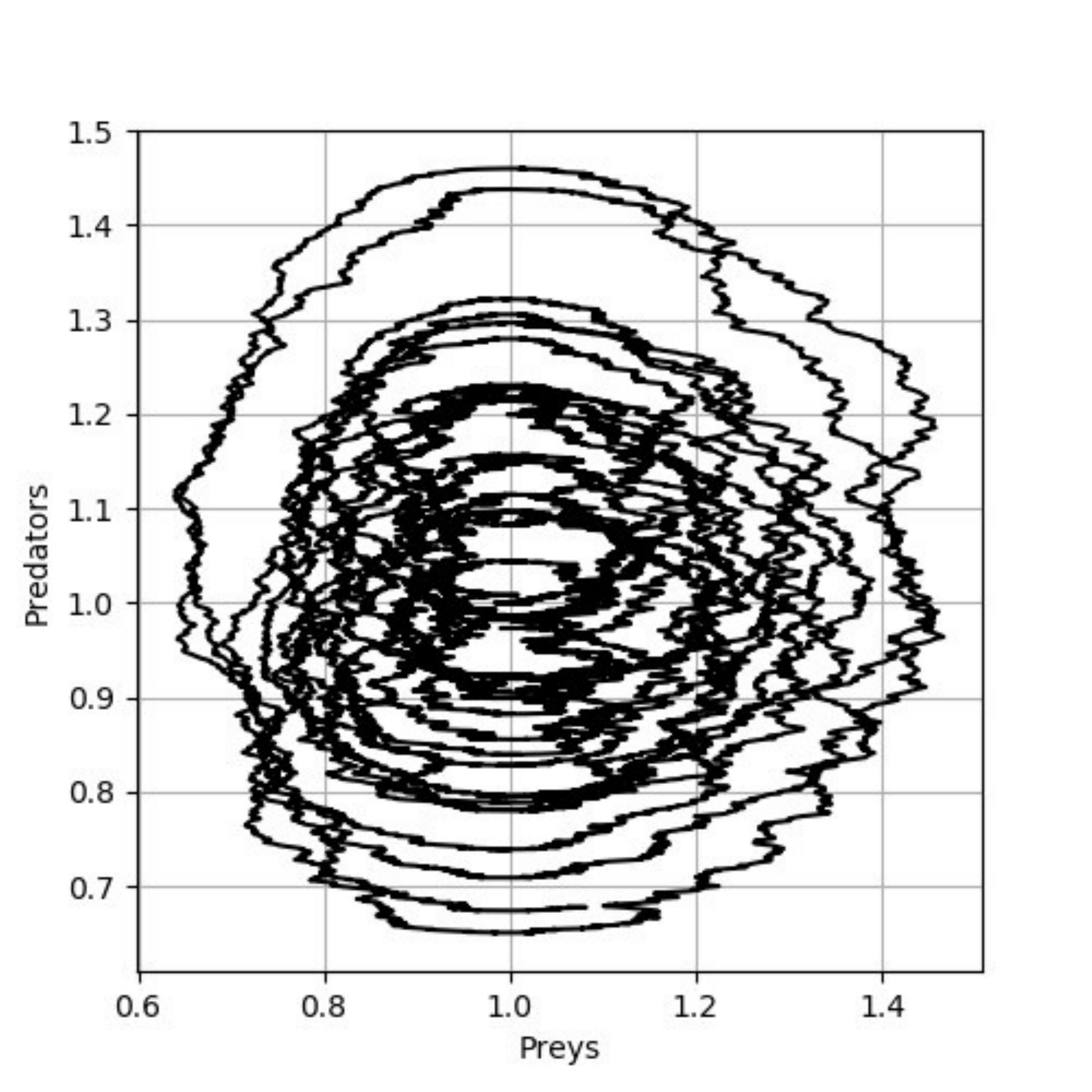}
		\caption{Typical phase trajectory of LV-system under fluctuating birthrate, starting from stationary point.}
		\label{image5}
	\end{figure}

	Then, we took ensemble of $M=100$ LV-systems, originating at the same stationary point (1,1), and found the mean squared displacement from this point as a function of time. The results for different time-steps and the same amplitude are shown at Fig.\ref{image6}.
	\begin{figure}[H]
		\centering
		\includegraphics[width=8cm,height=8cm]{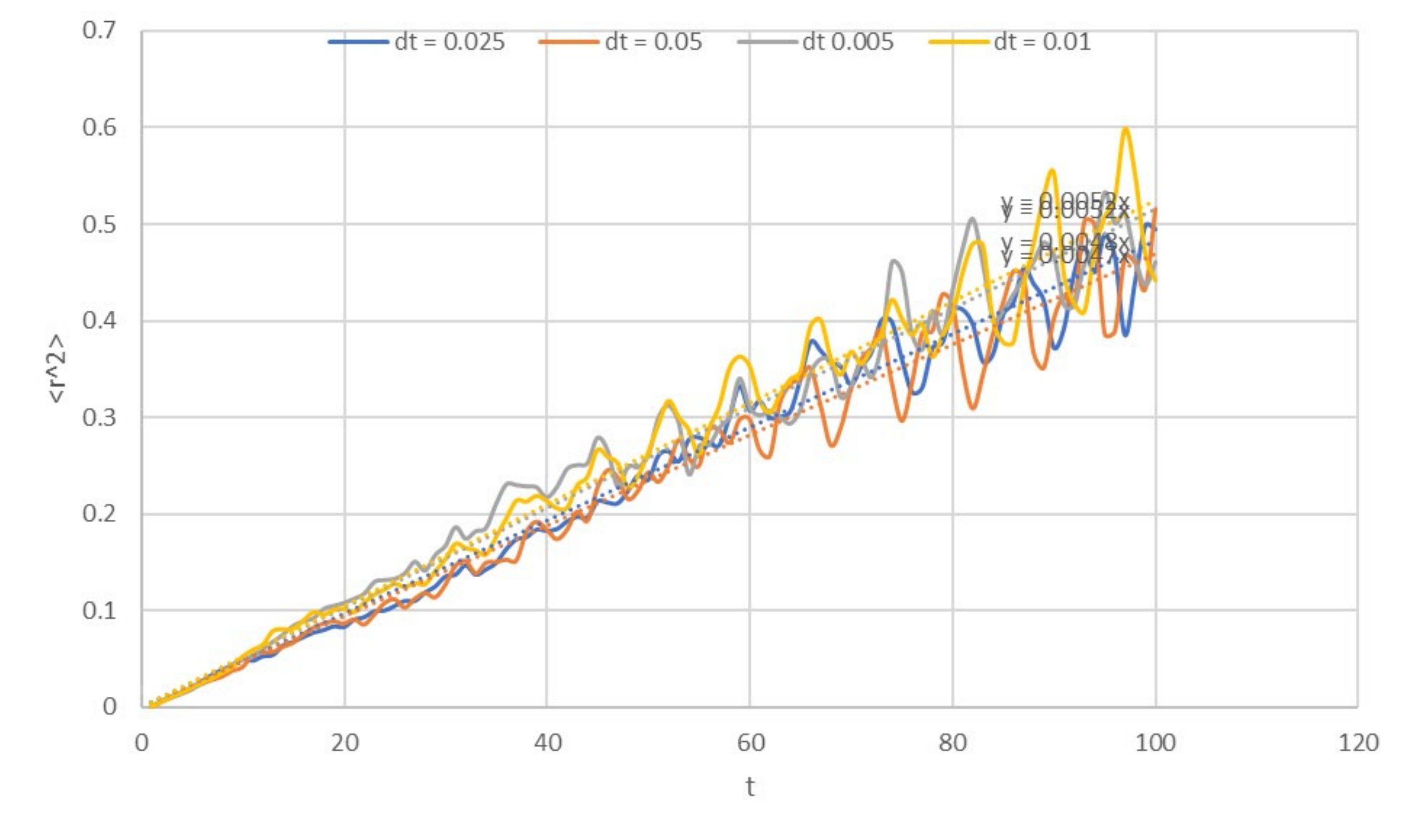}
		\caption{Noise of the birthrate k1. Mean squared displacement from stationary point versus time for the same amplitude A=0.07 and various time-step dt. In all cases trajectory starts from the stationary points}
		\label{image6}
	\end{figure}
	
	One can see that dependencies for different time-steps are close to each other and can be approximated as
	\begin{align}
	\tag{3.4}<(\Delta{X})^{2}+(\Delta{Y})^{2}>\approx{\alpha}A^{2}t, \alpha\approx1
	\label{eq3.4}
	\end{align}

	If the initial point differs from the stationary point, then the initial mean squared displacement at first decreases and then follows the same time law - see Fig.\ref{image7}.
	
	\begin{figure}[H]
		\centering
		\includegraphics[width=12cm,height=6cm]{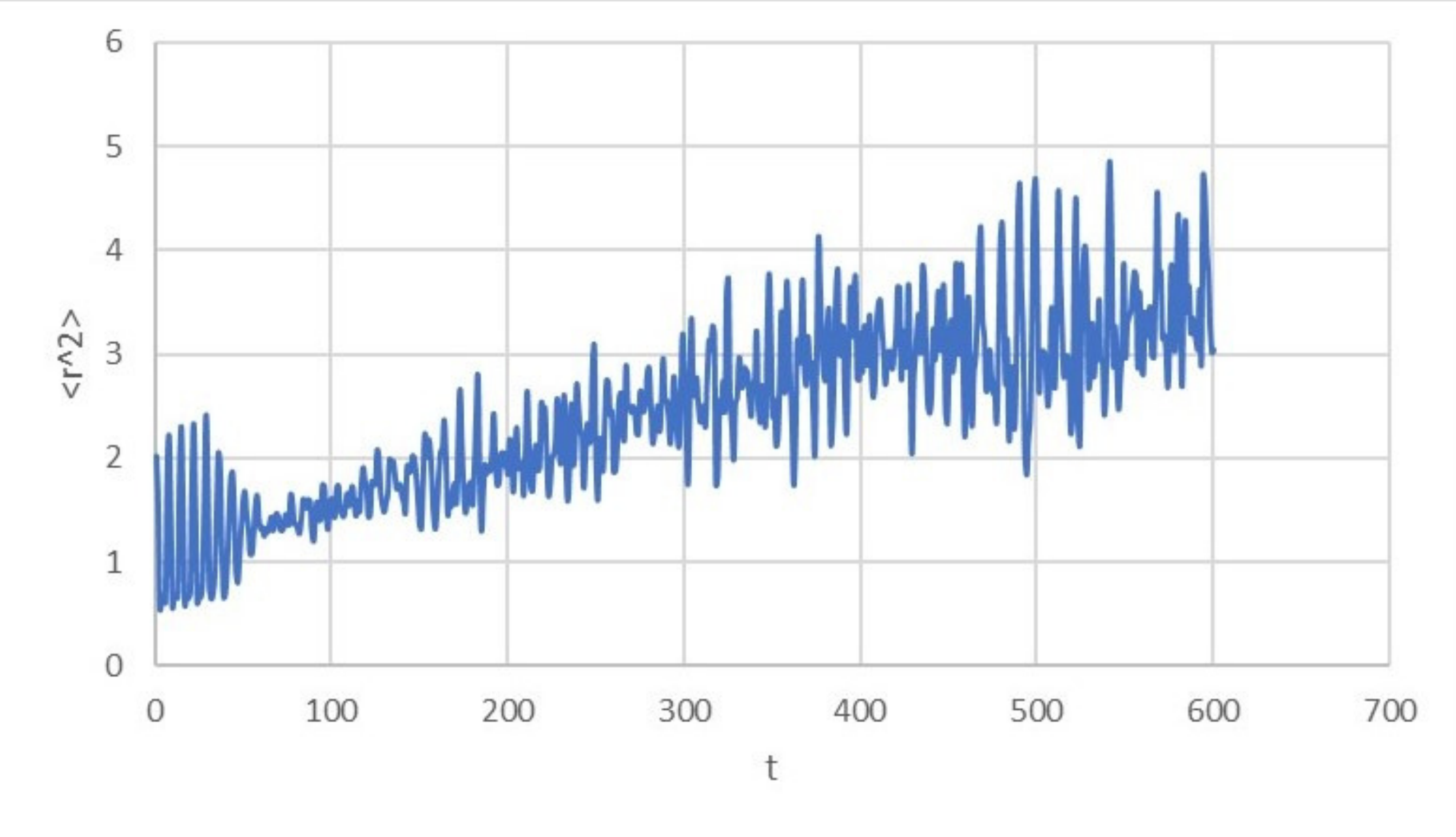}
		\caption{Noise of the birthrate k1.Mean squared displacement from stationary point versus time, with initial deviation from stationary point. ($A=0.07,dt=0.05,X_{0}=2,Y_{0}=2$)}
		\label{image7}
	\end{figure}
	
	Analogic dependencies were obtained by numeric simulations for the noise of $k_{2}, k_{3}, k_{4}$.
	
	Most probably parabolic dependence is related to ``zero stability'' of LV system, so that we observe something like random walk in XY space.
	
	We can prove Eq. (\ref{eq3.4}) analytically and find exact value of $\alpha$ at least for small deviations from steady-state using linearized kinetic equations.
	
	\begin{theorem}\label{th1}
		Linearization (first-order approximation) of the  Lotka-Volterra model in the vicinity of steady-state with the Langevin noise of the birthrate of preys and without external fluxes,
		\begin{align}
		\frac{dX}{dt}=(1+\xi(t))X-X\cdot{Y},\frac{dY}{dt}=X\cdot{Y}-Y,<\xi(t)\xi(t^{\prime})>=\nonumber\\
		=A^{2}\delta(t-t^{\prime}),\tag{3.5}
		\end{align}
		provides the following parabolic law for the sum of dispersions for the ensemble of LV-systems:
		\begin{align}
		<(\delta{X})^{2}>+<(\delta{Y})^{2}>=1\cdot{A}^{2}t\tag{3.6}
		\label{eq3.6}
		\end{align}
	\end{theorem}
	\begin{proof}
		Consider the first order approximation for the deviations from steady-state:
		$X=1+\delta{X},Y=1+\delta{Y},|\delta{X}|<<1,|\delta{Y}|<<1,$
	\begin{align}
	&\frac{d\delta{X}}{dt}=\xi(t)-\delta{Y}\tag{3.7a}\label{eq3.7a}\\
	&\frac{d\delta{Y}}{dt}=\delta{X}\tag{3.7b}\label{eq3.7b}
	\end{align}
	We consider an ensemble of LV-systems with identical initial condition\[\delta{X}(t=0)=0,\delta{Y}(t=0)=0\]
	Let us multiply Eq. (\ref{eq3.7a}) by $2\delta{X}$ and average over ensemble:
	\begin{align}
	\frac{d<(\delta{X})^{2}>}{dt}=-2<\delta{X}\delta{Y}>+2<\delta{X}(t)\xi(t)>\tag{3.8a}\label{eq3.8a}
	\end{align}
	Let us multiply Eq. (\ref{eq3.7b}) by $2\delta{Y}$ and average over ensemble:
	\begin{align}
	\frac{d<(\delta{Y})^{2}>}{dt}=2<\delta{X}\delta{Y}>\tag{3.8b}\label{eq3.8b}
	\end{align}
	Now we add Eqs (\ref{eq3.8a})+(\ref{eq3.8b}). It gives
	\begin{align}
	\frac{d}{dt}(<(\delta{X})^{2}>+<(\delta{Y})^{2}>)=2<\delta{X}(t)\xi(t)>\tag{3.9}
	\end{align}
	
	To make the set (\ref{eq3.8a}, \ref{eq3.8b}) self-consistent, one should find the value of $<\delta{X}(t)\xi(t)>$ .
	
	For this, first of all, we find the formal solution of set (\ref{eq3.7a}, \ref{eq3.7b}) which we reformulate in the matrix form:
	\begin{align}
	&\frac{d}{dt}
	\begin{pmatrix}
	\delta{X}\\
	\delta{Y}
	\end{pmatrix}=
	\begin{pmatrix}
	0& -1\\
	1& 0
	\end{pmatrix}
	\begin{pmatrix}
	\delta{X}\\
	\delta{Y}
	\end{pmatrix}+
	\begin{pmatrix}
	\xi\\
	0
	\end{pmatrix}
	\equiv
	\hat{M}
	\begin{pmatrix}
	\delta{X}\\
	\delta{Y}
	\end{pmatrix}+\hat{f},\nonumber\\
	&\delta{X}(t=0)=0,\delta{Y}(t=0)=0\tag{3.10}
	\label{eq3.10}
	\end{align}
	
	Solution of Eq. (\ref{eq3.10}) is:
	
	\begin{align}
	\begin{pmatrix}
	\delta{X}(t)\\
	\delta{Y}(t)
	\end{pmatrix}=
	\int_{0}^{t}\exp ((t-t^{\prime})\hat{M})
	\begin{pmatrix}
	\xi(t^{\prime})\\
	0
	\end{pmatrix}
	dt^{\prime}\tag{3.11}
	\end{align}
	
	Then we can find two mean values $<\delta{X}(t)\xi(t)>$ and $<\delta{Y}(t)\xi(t)>$.
	
	\begin{align}
	&<\begin{pmatrix}
	\delta{X}(t)\\
	\delta{Y}(t)
	\end{pmatrix}\xi(t)>=
	\begin{pmatrix}
	<\delta{X}(t)\xi(t)>\\
	<\delta{Y}(t)\xi(t)>
	\end{pmatrix}=
	\int_{0}^{t}\exp ((t-t^{\prime})\hat{M})\cdot\nonumber\\
	&\cdot<
	\begin{pmatrix}
	\xi(t^{\prime})\\
	0
	\end{pmatrix}
	\xi(t)>dt^{\prime}=\nonumber
	\int_{0}^{t}\exp ((t-t^{\prime})\hat{M})
	\begin{pmatrix}
	A^{2}\delta(t-t^{\prime})\\
	0
	\end{pmatrix}
	dt^{\prime}=\\
	&=\begin{pmatrix}
	A^{2}/2\\
	0
	\end{pmatrix}\tag{3.12}
	\end{align}
	
	(Note that the factor $1/2$ appeared because the argument of delta function $\delta(t-t^{\prime})$ is equal zero not within the integration interval but instead at it’s boundary.) Thus,
	\begin{align}
	&<\xi(t)\delta{X}(t)>=A^{2}/2,\nonumber\\
	&\frac{d}{dt}(<(\delta{X})^{2}>+<(\delta{Y})^{2}>)=A^{2}\Rightarrow<(\delta{X})^{2}>+<(\delta{Y})^{2}>=A^{2}t.\tag{3.13}
	\end{align}
	Theorem 1 is proved.
	\end{proof}
	Now we proceed to the noise of birthrate of predators.
	
	\begin{theorem}\label{th2}
		Linearization of the Lotka-Volterra model with the Langevin noise of the birthrate of predators and without external fluxes,
		\begin{align}
		\frac{dX}{dt}=X-X\cdot{Y},\frac{dY}{dt}=X\cdot{Y}-(1+\xi(t))Y,<\xi(t)\xi(t^{\prime})>=\nonumber\\
		=A^{2}\delta(t-t^{\prime}),\tag{3.14}
		\end{align}
		provides the parabolic law (\ref{eq3.6}) for the sum of dispersions.
	\end{theorem}

	\begin{proof}
		Again, consider the first order approximation for the deviations from steady-state:$X=1+\delta{X},Y=1+\delta{Y},|\delta{X}|<<1,|\delta{Y}|<<1,$
		\begin{align}
		&\frac{d\delta{X}}{dt}=-\delta{Y}\tag{3.15a}\label{eq3.15a}\\
		&\frac{d\delta{Y}}{dt}=\delta{X}-\xi(t)\tag{3.15b}\label{eq3.15b}
		\end{align}
		We again consider an ensemble of LV-systems with identical initial condition\[\delta{X}(t=0)=0,\delta{Y}(t=0)=0\]
		Let us multiply Eq. (\ref{eq3.15a}) by $2\delta{X}$ and average over ensemble:
		\begin{align}
		\frac{d<(\delta{X})^{2}>}{dt}=-2<\delta{X}\delta{Y}>\tag{3.16a}\label{eq3.16a}
		\end{align}
		Let us multiply Eq. (\ref{eq3.15b}) by $2\delta{Y}$ and average over ensemble:
		\begin{align}
		\frac{d<(\delta{Y})^{2}>}{dt}=2<\delta{X}\delta{Y}>-2<\delta{Y}\cdot\xi>\tag{3.16b}
		\label{eq3.16b}
		\end{align}
		Now we add Eqs (\ref{eq3.16a})+(\ref{eq3.16b}). It gives
		\begin{align}
		\frac{d}{dt}(<(\delta{X})^{2}>+<(\delta{Y})^{2}>)=-2<\delta{Y}(t)\xi(t)>\tag{3.17}
		\end{align}
		Now, one should find the value of $<\delta{Y(t)\xi(t)>}$.
		
		For this, first of all, we find the formal solution of set (\ref{eq3.15a}, \ref{eq3.15b}) which we reformulate in the matrix form:
		\begin{align}
		&\frac{d}{dt}
		\begin{pmatrix}
		\delta{X}\\
		\delta{Y}
		\end{pmatrix}=
		\begin{pmatrix}
		0& -1\\
		1& 0
		\end{pmatrix}
		\begin{pmatrix}
		\delta{X}\\
		\delta{Y}
		\end{pmatrix}+
		\begin{pmatrix}
		0\\
		-\xi
		\end{pmatrix}
		\equiv
		\hat{M}
		\begin{pmatrix}
		\delta{X}\\
		\delta{Y}
		\end{pmatrix}+\hat{f},\nonumber\\
		&\delta{X}(t=0)=0,\delta{Y}(t=0)=0\tag{3.18}
		\label{eq3.18}
		\end{align}
		Solution of Eq. (\ref{eq3.18}) is:
		\begin{align}
		\begin{pmatrix}
		\delta{X}(t)\\
		\delta{Y}(t)
		\end{pmatrix}=
		\int_{0}^{t}\exp ((t-t^{\prime})\hat{M})
		\begin{pmatrix}
		0\\
		-\xi(t^{\prime})
		\end{pmatrix}
		dt^{\prime}\tag{3.19}
		\end{align}
		Then we can find two mean values$<\delta{X}(t)\xi(t)>$ and $ < \delta Y( t )\xi ( t ) > $ .
		
		\begin{align}
		\begin{array}{l}
		< \left( {\begin{array}{*{20}{c}}
			{\delta X\left( t \right)}\\
			{\delta Y\left( t \right)}
			\end{array}} \right)\xi \left( t \right) >  = \left( {\begin{array}{*{20}{c}}
			{ < \delta X\left( t \right)\xi \left( t \right) > }\\
			{ < \delta Y\left( t \right)\xi \left( t \right) > }
			\end{array}} \right) = \\
			=\int\limits_0^t {\exp \left( {\left( {t - t'} \right)\hat M} \right)}  < \left( {\begin{array}{*{20}{c}}
			0\\
			{ - \xi \left( {t'} \right)}
			\end{array}} \right)\xi \left( t \right) > dt' = \\
		= \int\limits_0^t {\exp \left( {\left( {t - t'} \right)\hat M} \right)} \left( {\begin{array}{*{20}{c}}
			0\\
			{ - {A^2}\delta \left( {t - t'} \right)}
			\end{array}} \right)dt' = \left( {\begin{array}{*{20}{c}}
			0\\
			{ - {A^2}/2}
			\end{array}} \right)
		\end{array}
		\tag{3.20}
		\end{align}
		Thus,
		\begin{align}
		 < \delta Y\left( t \right)\xi \left( t \right) >  =  - {A^2}/2
		 \tag{3.21}
		\end{align}
		\begin{align}
		\frac{d}{{dt}}\left( { < {{\left( {\delta X} \right)}^2} >  + {{<\left( {\delta Y} \right)}^2}}> \right) = {A^2} \Rightarrow  < {\left( {\delta X} \right)^2} >  + {<\left( {\delta Y} \right)^2}> = {A^2}t
		\nonumber
		\end{align}
		Theorem is proved.	
	\end{proof}
	Now we proceed to the noise of extinction rate of preys (caught by predators).
	\begin{theorem}\label{th3}
		Linearization of the Lotka-Volterra model with the Langevin noise of extinction rate of preys and without external fluxes,
		\begin{align}
		\frac{{dX}}{{dt}} = X - \left( {1 + \xi \left( t \right)} \right)X \cdot Y,\,\,\frac{{dY}}{{dt}} = X \cdot Y - Y,\,\, < \xi \left( t \right)\xi \left( {t'} \right) >=\nonumber\\
		= {A^2}\delta \left( {t - t'} \right)
		\tag{3.22}
		\end{align}
		provides the parabolic law (\ref{eq3.6}) for the sum of dispersions.
	\end{theorem}
	\begin{proof}
		Again, consider the first order approximation for the deviations from steady-state: 
		$X = 1 + \delta X,\,\,Y = 1 + \delta Y,\,\,\,\left| {\delta X} \right| <  < 1,\,\,\left| {\delta Y} \right| <  < 1,$
		\begin{align}
		&\frac{{d\delta X}}{{dt}} =  - \xi \left( t \right) - \delta Y,\,\,\tag{3.23a}\label{eq3.23a}\\
		&\frac{{d\delta Y}}{{dt}} = \delta X
		\tag{3.23b}
		\label{eq3.23b}
		\end{align}
		Eqs. (\ref{eq3.23a}, \ref{eq3.23b}) are identical with eq. (\ref{eq3.7a}, \ref{eq3.7b}) except sign before noise term. All characteristics (probabilities and mean values) of Langevin noise are symmetrical in respect to sign. Therefore Theorem \ref{th3} directly follows from Theorem \ref{th1}.
	\end{proof}
	Now we proceed to the noise of reproduction rate of predators due to eating of preys.
	\begin{theorem}\label{th4}
		Linearization of the Lotka-Volterra model with the Langevin noise of reproduction rate of predators due to eating of preys and without external fluxes,
		\begin{align}
		\frac{{dX}}{{dt}} = X - X \cdot Y,\,\,\frac{{dY}}{{dt}} = \left( {1 + \xi \left( t \right)} \right)X \cdot Y - Y,\,\, < \xi \left( t \right)\xi \left( {t'} \right) >=\nonumber \\ 
		= {A^2}\delta \left( {t - t'} \right)
		\tag{3.24}
		\end{align}
		provides the parabolic law (\ref{eq3.6}) for the sum of dispersions.
	\end{theorem}
	\begin{proof}
		Again, consider the first order approximation for the deviations from steady-state: 
		$X = 1 + \delta X,\,\,Y = 1 + \delta Y,\,\,\,\left| {\delta X} \right| <  < 1,\,\,\left| {\delta Y} \right| <  < 1\,,$
		\begin{align}
		&\frac{{d\delta X}}{{dt}} =  - \delta Y,\,\,\tag{3.25a}\label{eq3.25a}\\
		&\frac{{d\delta Y}}{{dt}} = \delta X + \xi \left( t \right)
		\tag{3.25b}\label{eq3.25b}
		\end{align}
		Eqs. (\ref{eq3.25a}, \ref{eq3.25b}) are identical with eq. (\ref{eq3.15a}, \ref{eq3.15b}) except sign before noise term. All characteristics (probabilities and mean values) of Langevin noise are symmetrical in respect to sign. Therefore, Theorem \ref{th4} directly follows from Theorem \ref{th2}.
	\end{proof}
	Now we proceed to the noise in the LV-model with external fluxes – adding noise to various terms in eqs.(\ref{eq4}). We will start with the simplest cases – when the external flux is present but it’s mean value is zero.
	\begin{theorem}\label{th5}
		Linearization of the Lotka-Volterra model with the Langevin noise of external flux of preys,
		\begin{align}
		\frac{{dX}}{{dt}} = X - X \cdot Y + \xi \left( t \right),\,\,\frac{{dY}}{{dt}} = X \cdot Y - Y,\,\, < \xi \left( t \right)\xi \left( {t'} \right) >=\nonumber\\
		 = {A^2}\delta \left( {t - t'} \right)
		\tag{3.26}
		\end{align}
		provides the parabolic law (\ref{eq3.6}) for the sum of dispersions.
	\end{theorem}
	\begin{proof}
		Again, consider the first order approximation for the deviations from steady-state: 
		$X = 1 + \delta X,\,\,Y = 1 + \delta Y,\,\,\,\left| {\delta X} \right| <  < 1,\,\,\left| {\delta Y} \right| <  < 1\,,$
		\begin{align}
		&\frac{{d\delta X}}{{dt}} =  - \delta Y + \xi \left( t \right),\,\,\tag{3.27a}\label{eq3.27a}\\
		&\frac{{d\delta Y}}{{dt}} = \delta X
		\tag{3.27b}\label{eq3.27b}
		\end{align}
		Eqs.(\ref{eq3.27a}, \ref{eq3.27b}) are identical with (\ref{eq3.7a}, \ref{eq3.7b}). Therefore, Theorem \ref{th5} directly follows from Theorem \ref{th1}.
	\end{proof}
	\begin{theorem}\label{th6}
		Linearization of the Lotka-Volterra model with the Langevin noise of external flux of predators,
		\begin{align}
		\frac{{dX}}{{dt}} = X - X \cdot Y,\,\,\frac{{dY}}{{dt}} = X \cdot Y - Y + \xi \left( t \right),\,\, < \xi \left( t \right)\xi \left( {t'} \right) >=\nonumber\\  = {A^2}\delta \left( {t - t'} \right)
		\tag{3.28}
		\end{align}
		provides the parabolic law (\ref{eq3.6}) for the sum of dispersions.
	\end{theorem}
	\begin{proof}
		Again, consider the first order approximation for the deviations from steady-state: 
		$X = 1 + \delta X,\,\,Y = 1 + \delta Y,\,\,\,\left| {\delta X} \right| <  < 1,\,\,\left| {\delta Y} \right| <  < 1\,$
		\begin{align}
		&\frac{{d\delta X}}{{dt}} =  - \delta Y,\,\,\tag{3.29a}\label{eq3.29a}\\
		&\frac{{d\delta Y}}{{dt}} = \delta X + \xi \left( t \right)
		\tag{3.29b}\label{eq3.29b}
		\end{align}
		Eqs.(\ref{eq3.29a}, \ref{eq3.29b}) are identical with (\ref{eq3.25a}, \ref{eq3.25b}). Therefore, Theorem \ref{th6} directly follows from Theorem \ref{th4}.
	\end{proof}
	Now we proceed to the noise in LV-systems with non-zero mean external fluxes. Of course, in this case the steady-state reduced values are not equal to 1 anymore, but instead determined by eq. (\ref{eq5}). In this paper we consider only one case - the Lottka-Volterra model with the Langevin noise of the birthrate
	\begin{align}
	&\frac{{dX}}{{dt}} = \left( {1 + \xi \left( t \right)} \right)X - X \cdot Y + {b_x},\,\,\frac{{dY}}{{dt}} = X \cdot Y - Y + {b_y},\,\,\nonumber\\
	& < \xi \left( t \right)\xi \left( {t'} \right) >  = {A^2}\delta \left( {t - t'} \right)
	\tag{3.30}
	\end{align}
	Linearization for first-order deviations from steady-state 
	($X = {X^{st}} + \delta X,\,\,Y = {Y^{st}} + \delta Y,\,\,\,\left| {\delta X} \right| <  < {X^{st}},\,\,\left| {\delta Y} \right| <  < {Y^{st}}\,$), gives
	\begin{align}
	&\frac{{d\delta X}}{{dt}} = \xi \left( t \right){X^{st}} + \left( {1 - {Y^{st}}} \right)\delta X - {X^{st}}\delta Y,\,\,\tag{3.31a}\label{eq3.31a}\\
	&\frac{{d\delta Y}}{{dt}} = {Y^{st}}\delta X + \left( {{X^{st}} - 1} \right)\delta Y
	\tag{3.31b}\label{eq3.31b}
	\end{align}
	We again consider an ensemble with identical initial condition $\delta X\left( {t = 0} \right) = 0,\,\,\delta Y\left( {t = 0} \right) = 0$
	Let us multiply eq. (\ref{eq3.31a}) by $2\delta X$ and average over ensemble:
	\begin{align}
	\frac{{d < {{\left( {\delta X} \right)}^2} > }}{{dt}} =  + 2\left( {1 - {Y^{st}}} \right) < {\left( {\delta X} \right)^2} >  - 2{X^{st}} < \delta X\delta Y >+\nonumber\\  + 2{X^{st}} < \delta X\left( t \right)\xi \left( t \right) > 
	\tag{3.32a}
	\label{eq3.32a}
	\end{align}
	Let us multiply eq. (\ref{eq3.31b}) by $2\delta Y$ and average over ensemble:
	\begin{align}
	\frac{{d < {{\left( {\delta Y} \right)}^2} > }}{{dt}} = 2\left( {{X^{st}} - 1} \right) < {\left( {\delta Y} \right)^2} >  + 2{Y^{st}} < \delta X\delta Y > 
	\tag{3.32b}
	\label{eq3.32b}
	\end{align}
	Let us add the product of eq. (\ref{eq3.31a}) with $2\delta Y$ and the product of eq. (\ref{eq3.31b}) with $2\delta X$ and average over ensemble:
	\begin{align}
	\frac{{d < \delta X\delta Y > }}{{dt}} = {Y^{st}} < {\left( {\delta X} \right)^2} >  - {X^{st}} < {\left( {\delta Y} \right)^2} >+\nonumber\\  + \left( {{X^{st}} + {Y^{st}} - 2} \right) < \delta X\delta Y >  + < \delta Y\left( t \right)\xi \left( t \right) > 
	\tag{3.32c}
	\label{eq3.32c}
	\end{align}
	To make the set (\ref{eq3.32a}, \ref{eq3.32b}, \ref{eq3.32c}) self-consistent, one should find the values of $ < \delta X\left( t \right)\xi \left( t \right) > $, $ < \delta Y\left( t \right)\xi \left( t \right) > $.
	
	For this, first of all, we find the formal solution of set (\ref{eq3.7a}, \ref{eq3.7b}) which we reformulate in the matrix form:
	\begin{align}
	\frac{d}{{dt}}\left( {\begin{array}{*{20}{c}}
		{\delta X}\\
		{\delta Y}
		\end{array}} \right) = \left( {\begin{array}{*{20}{c}}
		{\left( {1 - {Y^{st}}} \right)}&{ - {X^{st}}}\\
		{{Y^{st}}}&{\left( {{X^{st}} - 1} \right)}
		\end{array}} \right)\left( {\begin{array}{*{20}{c}}
		{\delta X}\\
		{\delta Y}
		\end{array}} \right) + \left( {\begin{array}{*{20}{c}}
		\xi \\
		0
		\end{array}} \right) \equiv\nonumber\\
		\equiv \hat M\left( {\begin{array}{*{20}{c}}
		{\delta X}\\
		{\delta Y}
		\end{array}} \right) + \hat f,\nonumber\\
	\delta X\left( {t = 0} \right) = 0,\,\,\delta Y\left( {t = 0} \right) = 0
	\nonumber
	\end{align}
	Solution is:
	\begin{align}
	\left( {\begin{array}{*{20}{c}}
		{\delta X\left( t \right)}\\
		{\delta Y\left( t \right)}
		\end{array}} \right) = \int\limits_0^t {\exp \left( {\left( {t - t'} \right)\hat M} \right)} \left( {\begin{array}{*{20}{c}}
		{\xi \left( {t'} \right)}\\
		0
		\end{array}} \right)dt'
	\nonumber
	\end{align}
	Then we can find two mean values $ < \delta X\left( t \right)\xi \left( t \right) > $ and $ < \delta Y\left( t \right)\xi \left( t \right) > $.
	\begin{align}
	< \left( {\begin{array}{*{20}{c}}
		{\delta X\left( t \right)}\\
		{\delta Y\left( t \right)}
		\end{array}} \right)\xi \left( t \right) >  = \left( {\begin{array}{*{20}{c}}
		{ < \delta X\left( t \right)\xi \left( t \right) > }\\
		{ < \delta Y\left( t \right)\xi \left( t \right) > }
		\end{array}} \right)=\nonumber\\ = \int\limits_0^t {\exp \left( {\left( {t - t'} \right)\hat M} \right)}  < \left( {\begin{array}{*{20}{c}}
		{\xi \left( {t'} \right)}\\
		0
		\end{array}} \right)\xi \left( t \right) > dt' = \nonumber\\
	= \int\limits_0^t {\exp \left( {\left( {t - t'} \right)\hat M} \right)} \left( {\begin{array}{*{20}{c}}
		{{A^2}\delta \left( {t - t'} \right)}\\
		0
		\end{array}} \right)dt' = \left( {\begin{array}{*{20}{c}}
		{{A^2}/2}\\
		0
		\end{array}} \right)
	\tag{3.33}
	\end{align}
	Thus, the set of equations (\ref{eq3.32a}, \ref{eq3.32b}, \ref{eq3.32c}) becomes self-consistent:
	\begin{align}
	&\frac{d}{{dt}}\left( {\begin{array}{*{20}{c}}
		{ < {{\left( {\delta X} \right)}^2} > }\\
		{ < {{\left( {\delta Y} \right)}^2} > }\\
		{ < \left( {\delta X\delta Y} \right) > }
		\end{array}} \right)=\nonumber\\
		& = \left( {\begin{array}{*{20}{c}}
		{2\left( {1 - {Y^{st}}} \right)}&0&{ - 2{X^{st}}}\\
		0&{2\left( {{X^{st}} - 1} \right)}&{2{Y^{st}}}\\
		{{Y^{st}}}&{ - {X^{st}}}&{{X^{st}} + {Y^{st}} - 2}
		\end{array}} \right)\left( {\begin{array}{*{20}{c}}
		{ < {{\left( {\delta X} \right)}^2} > }\\
		{ < {{\left( {\delta Y} \right)}^2} > }\\
		{ < \left( {\delta X\delta Y} \right) > }
		\end{array}} \right) +\nonumber\\
		&+ \left( {\begin{array}{*{20}{c}}
		{2{X^{st}}{A^2}/2}\\
		0\\
		0
		\end{array}} \right)
	\tag{3.34}
	\label{eq3.34}
	\end{align}
	The formal solution (in matrix form) of eq. (\ref{eq3.34}) is:
	\begin{align}
	\hat \psi \left( t \right) = \int\limits_0^t {\exp \left( {\left( {t - t'} \right)\hat L} \right)} \hat \varphi \left( {t'} \right)dt',\,\,\hat \psi \left( {t = 0} \right) = 0
	\tag{3.35}
	\end{align}
	where
	\begin{align}
	\hat \psi \left( t \right) = \left( {\begin{array}{*{20}{c}}
		{ < {{\left( {\delta X} \right)}^2} > }\\
		{ < {{\left( {\delta Y} \right)}^2} > }\\
		{ < \left( {\delta X\delta Y} \right) > }
		\end{array}} \right),\nonumber\\
		\hat L = \left( {\begin{array}{*{20}{c}}
		{2\left( {1 - {Y^{st}}} \right)}&0&{ - 2{X^{st}}}\\
		0&{2\left( {{X^{st}} - 1} \right)}&{2{Y^{st}}}\\
		{{Y^{st}}}&{ - {X^{st}}}&{{X^{st}} + {Y^{st}} - 2}
		\end{array}} \right),\nonumber\\
		\hat \varphi  = \left( {\begin{array}{*{20}{c}}
		{{X^{st}}{A^2}}\\
		0\\
		0
		\end{array}} \right)
	\tag{3.36}
	\end{align}
	
	If the steady-state solution satisfies the local stability criterion (${\mathop{\rm Im}\nolimits} {\lambda _1} < 0,\,{\mathop{\rm Im}\nolimits} {\lambda _2} < 0$ in eqs. (\ref{eq2.4})), then the set (\ref{eq3.34}) should describe the competition between dynamical tendency of attraction to the metastable steady-state and stochastic tendency of migration from this steady state. It is physically evident that sooner or later the stochastic (noise) will bring the system beyond the limits of metastability (beyond the convergence region). It is also evident that the “Mean Time To Failure” (MTTF) of the LV-system should depend on the noise amplitude. This problem will be considered elsewhere. (See also \cite{lit10}.)
	\section*{Conclusion}
	\begin{enumerate}
	\item 	Regular hunting of preys without troubling predators destabilizes the LV-system.
	\item 	Regular hunting of predators without troubling preys destabilizes the LV-system.
	\item 	Regular adding of preys always stabilizes the LV-system.
	\item	Regular adding of predators stabilizes the LV-system only if adding rate less than some threshold.
	\item Diagram of regimes for LV-system with regular external fluxes is described by Fig.\ref{image3}.
	\item	Simultaneous adding of preys and limited hunting of predators may leave the system metastable but only within some critical region of initial parameters.
	\item	In numeric modeling, to make the impact of the noise of kinetic parameters independent on the choice of time step, the random perturbation kinetic coefficients should be proportional to noise amplitude and inversely proportional to the square root of the time step.
	\item	Mean square distance from stationary point (at least in case of small deviations from steady-state) increases proportionally to time with proportionality coefficient proportional to the noise amplitude:$<(\Delta{X})^{2}+(\Delta{Y})^{2}>={A}^{2}t$
	\item	Competition between noise and metastability under non-zero external fluxes (including MTTF) will be analyzed elsewhere.
	\end{enumerate}

	\bigskip
	
	CONTACT INFORMATION
	
	\medskip
	Yaroslav Huriev - third-year student of  
	Educational-Scientific Institute of
	Informational and Eduational Technologies,
	The Bohdan Khmelnytsky National University of Cherkasy\\
	
	Andriy Gusak - leading researcher of the Laboratory of Mathematical Physics of the Bohdan Khmelnytsky National University of Cherkasy/Institute of Applied Mathematics and Mechanics, professor of the physics department of the Bohdan Khmelnytsky National University of Cherkasy, doctor of physics and mathematics sciences\\ \\
	Laboratory of Mathematical Physics,\\
	(joint lab. of Cherkasy National University and 
	Institute of Applied Mathematics and Mechanics)\\
	81 Shevchenko blvd.,\\
	18000 Cherkasy,\\ 
	Ukraine
	\\
	E-mail: \href{mailto:yaroslavhuriev@gmail.com}{yaroslavhuriev@gmail.com},\\
	\href{mailto:amgusak@ukr.net}{amgusak@ukr.net}
	
\end{document}